\documentclass[11pt]{amsart}

\usepackage[colorlinks,citecolor=magenta,linkcolor=black]{hyperref}
\pdfpagewidth=\paperwidth \pdfpageheight=\paperheight
\usepackage{amsfonts,amssymb,amsthm,amsmath,eucal,tabu,url}
\usepackage{pgf}
 \usepackage{array}
 \usepackage{tikz-cd}
 \usepackage{pstricks}
 \usepackage{pstricks-add}
 \usepackage{pgf,tikz}
 \usetikzlibrary{automata}
 \usetikzlibrary{arrows}
 \usepackage{indentfirst}
\usepackage{tabularx} 
\theoremstyle{plain}
\newtheorem{thm}{Theorem}[section]
\newtheorem{theorem}[thm]{Theorem}

\newtheorem{lemma}[thm]{Lemma}
\newtheorem{proposition}[thm]{Proposition}
\newtheorem{corollary}[thm]{Corollary}

\theoremstyle{definition}
\newtheorem{definition}[thm]{Definition}
\newtheorem{remark}[thm]{Remark}
\newtheorem{example}[thm]{Example}

\numberwithin{equation}{section}
\DeclareMathOperator{\mult}{mult}

\begin{document}
\title[Hunting for Miyaoka-Kobayashi curves]{Hunting for Miyaoka-Kobayashi curves}

\author[Alexandru Dimca]{Alexandru Dimca}
\address{Universit\'e C\^ ote d'Azur, CNRS, LJAD, France and Simion Stoilow Institute of Mathematics,
P.O. Box 1-764, RO-014700 Bucharest, Romania.}
\email{dimca@unice.fr}

\author[Piotr Pokora]{Piotr Pokora}
\address{Department of Mathematics,
University of the National Education Commission Krakow,
Podchor\c a\.zych 2,
PL-30-084 Krak\'ow, Poland.}
\email{piotr.pokora@up.krakow.pl}

\thanks{\vskip0\baselineskip
\vskip-\baselineskip
\noindent Communicated by Y. Namikawa. Received February 17, 2023. Revised November 9, 2023.}

\subjclass[2020]{Primary: 14H50; Secondary:  14N20, 32S35}
\maketitle

\thispagestyle{empty}
\begin{abstract}
 We study the possible singularities of Miyaoka-Kobayashi curves and prove several results about the non-existence of such curves in degrees $\geq 8$. \\ 
 \textbf{Keywords}: Miyaoka-Kobayashi numbers, MK-curves, free curves, ADE-singularities.
\end{abstract}

\maketitle
\section{Introduction}

In the present paper we recall and study, after years of oblivion, the class of Miyaoka-Kobayashi curves (MK-curves for short). These  complex projective plane curves of even degree $n \geq 6$ were defined and studied in Hirzebruch's paper \cite{Hirzebruch} and in his student Ivinskis' Master Thesis \cite{Ivinskis}.
They are important for constructing complex ball-quotient surfaces - the related facts are briefly recalled in Section $2$ below as a motivation for our paper.

In Section $3$, we rephrase firstly the definition of MK-curves in terms of total Tjurina numbers, see Corollary \ref{C1}. Then, using deep properties of these Tjurina numbers going back to results by du Plessis and Wall \cite{duP}, we show in Theorem \ref{T1} that an MK-curve must have a lot of singularities. Finally, we quote Ivinskis' result saying that in degree 6, there are exactly three types of MK-curves, see Theorem \ref{thmI}. 

On the other hand, in degrees $n \geq 8$ no example of MK-curve is known, and most of our results give restrictions on the singularities that an MK-curve might have.
For instance, in Section 4, we get restrictions on the singularities of an irreducible MK-curve using the dual curves.

In Section $5$, we prove a number of non-existence results for MK-curves having only some classes of ADE singularities, and this is done using Langer's version on orbifold Bogomolov-Miyaoka-Yau inequality from \cite{Langer}. The main results here are
Theorems \ref{thmn=8}, \ref{thmE6}, \ref{thmE8} and Corollaries
\ref{corn=10}, \ref{corA1}, \ref{corA2}.

Finally, in Section $6$ we try to understand if there is a relation between MK-curves and free curves. This question is motivated by the fact that
in degree $6$, where we have three MK-curves, one of them is a free curve, and the other two are nearly free, see Example \ref{exNU} for details.
We recall the definition of the freeness defect $\nu(C)$ for any reduced plane curve and give a formula for $\nu(C)$ in terms of the total Tjurina number $\tau(C)$ when $C$ has only ADE singularities, see Theorem \ref{thmNU}. In Example \ref{exNU2}, we compute the freeness defect $\nu(C)$ for two curves of degree $18$, constructed by Bonnaf\'e in
\cite{Bonnafe}, and which are in some sense close to being MK-curves.
This computation shows that one of them is nearly free, while the other has no special freeness properties.

\section{Why do we care about MK-curves?}

Let us describe first the setting for the constructions of complex ball-quotient surfaces using MK-curves.

\begin{definition}
\label{D1}
Let $X$ be a complex quasi-smooth projective surface, i.e. a surface having only isolated quotient singularities. For a quotient singularity  $p \in {\rm Sing}(X)$ we denote by $G_{p}$ the associated small group. Denote by $\pi : \tilde{X} \rightarrow X$ the minimal resolution of singularities of $X$ and by $E:=\sum_{p \in {\rm Sing}(X)}E_{p}$ the union of all exceptional curves over $p \in {\rm Sing}(X)$. The rational number
$$m(p)=m(E_{p}):=3(e(E_{p})-1/{\rm ord}(G_{p})) - K(E_{p})^{2},$$
where $K(E_{p})$ denotes the local canonical divisor of $p$ and $e(E_{p})$ denotes the Euler characteristic, is called the Miyaoka-Kobayashi number of the quotient singularity $p \in {\rm Sing}(X)$.
\begin{remark}
For the convenience of the reader, let us recall the definition of the local canonical divisor $K(E_{p})$. Let $X$ be a normal projective surface and let $\pi : \tilde{X} \rightarrow X$ be a minimal resolution of singularities of $X$. For $p \in X$ being a singular point, let $E_{p} = \sum_{i=1}^{r}E_{i}$ be the associated exceptional curve. Now the local canonical divisor of $p\in X$, denoted by $K(E_{p}) = \sum_{i=1}^{r}c_{i}E_{i}$, is a $\mathbb{Q}$-divisor uniquely defined by the condition
$$K(E_{p}).E_{i} = K_{\tilde{X}}.E_{i} \quad \text{ for all } \quad i \in \{1, ..., r\}.$$
\end{remark}
\end{definition}
The following result is related to simply singular curves, i.e., curves having ${\rm ADE}$ singularities, in the setting of MK numbers, \cite{Hirzebruch}.
\begin{theorem}[Hirzebruch-Ivinskis]
\label{thm1}
Let $C \subset \mathbb{P}^{2}_{\mathbb{C}}$ be a reduced simply singular curve of even degree $n\geq 6$. Then one has
$$\sum_{p \in {\rm Sing}(C)}m(p)\leq \frac{n}{2}(5n-6),$$
where $m(p)$ denotes the Miyaoka-Kobayashi number of the double cover $X$ of  $\mathbb{P}^{2}_{\mathbb{C}}$ ramified along $C$ at the point $p\in X$, identified to the singularity $p \in C$.
\end{theorem}

\begin{definition}
\label{D2}
Let $C \subset \mathbb{P}^{2}_{\mathbb{C}}$ be a reduced simply singular curve of even degree $n\geq 6$. Define $m(C):=\sum_{p \in {\rm Sing}(C)}m(p)$. We say that $C$ is an MK-curve if
$$m(C) = \frac{n}{2}(5n-6).$$
\end{definition}
The notion of MK-curves plays a vital role for constructing ball-quotient surfaces and in order to see this, let us present here a general setting, and then we explain the link with the setting of our paper. Let $X$ be a quasi-smooth projective surface and for $p \in {\rm Sing}(X)$ denote by $G_{p}$ the associated small group. Consider $\pi: \tilde{X} \rightarrow X$ the minimal resolution of singularities and $E = \sum_{p \in {\rm Sing}(X)}E_{p}$ the union of the exceptional curves over $p \in {\rm Sing}(X)$. Consider now $D = \sum_{i}D_{i}$ a curve in $\tilde{X}$ with normal crossings and such that the connected components  $D_{i}$ of $D$ can be written as  $D_{i}=\sum D_{ij}$, where $D_{ij}$ are irreducible and distinct. Assume  that ${\rm supp}\, D \cap {\rm supp} \, E = \emptyset$. We are writing $(X,\tilde{X},D,E)$ for such a $4$-tuple. One allows to have $D$ or $E$ as zero divisors.
For such $4$-tuples Kobayashi proved the following result - see \cite[3.2.3 Theorem]{Ivinskis} and \cite{Kob1, Kob2}.

\begin{theorem} 
\label{ball-quotient}
Let $(X,\tilde{X},D,E)$ be a $4$-tuple  as above and assume that the Kodaira dimension of $K_{\tilde{X}}+D+E$ is equal to $2$. Assume furthermore the following conditions
\begin{enumerate}
\item[1)] for each irreducible component $D_{ij}$ of $D$ we have $(K_{\tilde{X}}+D).D_{ij}\geq 0$, and we have strict inequality when $D_{ij}$ is smooth;
\item[2)] there does not exist any smooth rational $(-1)$-curve such that $D.C\leq 1$;
\item[3)] if $C$ is a smooth rational $(-2)$-curve with ${\rm supp}\, C \cap {\rm supp}\, D = \emptyset$, ${\rm supp}\, C \not\subset {\rm supp}\, E$ and $C . E >0$, then there exists a singular point $p \in {\rm Sing}(X)$ with $E_{p}.C > 0$ with the property that $p$ is not a rational double point. 
\end{enumerate}
Then one has the following
\begin{equation}
\label{MKB}
\sum_{p \in {\rm Sing}(X)}m(E_{p}) + \sum_{i}n(D_{i}) \leq 3c_{2}(\tilde{X}) -c_{1}^{2}(\tilde{X}),
\end{equation}
where $n(D_{i}) :=3e(D_{i})+2K_{\tilde{X}}.D_{i}+D_{i}^{2}$, and the equality holds exactly when the universal cover of $\tilde{X} \setminus {\rm supp}(D+E)$ is biholomorphic with $\{(z_{1},z_{2})\in \mathbb{C}^{2} \,:\, |z_{1}|^{2}+|z_{2}|^{2}<1\} \setminus \{ \text{a discrete set of points}\}.$

\end{theorem}
In other words, the inequality in Theorem \ref{ball-quotient} holds exactly when there exists a discrete subgroup $\Gamma \subset {\rm PSU}(2,1)$ which acts properly discontinuously on the ball $\mathbb{B}^{2}:=\{(z_{1},z_{2})\in \mathbb{C}^{2} \,:\, |z_{1}|^{2}+|z_{2}|^{2}<1\}$ having only isolated fix points such that $X \setminus \pi({\rm supp}\, D) \cong \Gamma / \mathbb{B}^{2}$. Based on that remark, we can define the following fundamental object of studies.
\begin{definition}
A $4$-tuple $(X, \tilde{X}, D,E)$ is called a ball-quotient provided that $(X,\tilde{X},D,E)$ satisfies the assumptions of Theorem \ref{ball-quotient} and the equality holds in (\ref{MKB}).
\end{definition}

Based on the above considerations, it is worth pointing out that the mentioned result by Hirzebruch and Ivinskis follows from a direct application of Theorem \ref{ball-quotient} to the following setting: $f: X \rightarrow \mathbb{P}^{2}_{\mathbb{C}}$ is the double cover branched along a simply singular curve of degree $n\geq 6$, $\tilde{X}$ is the minimal resolution of singularities of $X$, $D=0$, and $E = \sum_{p \in {\rm Sing}(X)}E_{p}$. 

If we think in the setting of $4$-tuples $(X,\tilde{X},0,E)$, recall that we have $q(\tilde{X})=0$. If now $n=6$, then $\tilde{X}$ is a $K3$ surface. It turns our that MK-sextic curves are special in the sense of constructed algebraic surfaces, and this observation was summed up and presented by Kobayashi in 1985 during his lectures \cite{Koba}.

\begin{theorem} 
Let $X$ be an algebraic surface having only rational double points as singularities and let $\widetilde{X}$ be a $K3$ surface obtained as the  minimal resolution $\pi: \widetilde{X} \rightarrow X$. Then $\sum_{p \in {\rm Sing}(X)}m(p) = 72$ exactly when $X$ is a quotient of the complex two dimensional torus by a finite group of automorphisms of the torus. 
\end{theorem}
Note also that there are essentially three MK-sextics, see Theorem \ref{thmI} due to Ivinskis. 

On the other hand, no example of MK-curve is known in degrees $\geq 8$. It is well-known that if $C$ is a simply singular curve of degree $n\geq 8$, then the associated surface is of general type, and if furthermore $C$ is an MK-curve, then $(X,\tilde{X},0,E)$ is a ball-quotient. 
This brief presentation shows the significance of our search for  MK-curves  in degrees $\geq 8$. 

\section{MK-curves have many singularities}

Here we consider MK-curves in the complex projective plane,  and first we restate definitions regarding these curves using the total Tjurina number of a given reduced curve $C \subset \mathbb{P}^{2}_{\mathbb{C}}$. We use the same notation $A,D,E$ for the singular point $p$, when we look at it as a singular point on $X$ or on $C$. This is convenient since, for instance, the Tjurina numbers are the same, namely $\tau(X,p)=\tau(C,p)$.
\begin{lemma} We have the following formulas on Miyaoka-Kobayashi numbers of singular points:
\label{L1}
$$m(A_k)=3(k+1)-\frac{3}{k+1}  \text{ for } k \geq 1,$$
$$m(D_{k})=3(k+1)- \frac{3}{4(k-2)} \text{ for } k \geq 4,$$
$$m(E_{6})= \frac{167}{8}, \   
m(E_{7})= \frac{383}{16},\text{ and } 
m(E_{8})= \frac{1079}{40}.$$
\end{lemma}
This result can be easily reformulated as follows.
\begin{lemma}
\label{L2}
For any isolated hypersurface singularity $W$ we denote by $\tau(W)$ the Tjurina number of the singularity. Then we have the following.
$$\frac{m(A_k)}{3}= \tau(A_k)+\epsilon(A_k), \text{ where }\epsilon(A_k)=\frac{k}{k+1}  \text{ for } k \geq 1,$$
$$   \frac{m(D_k)}{3}= \tau(D_k)+\epsilon(D_k), \text{ where }\epsilon(D_k)=\frac{4k-9}{4(k-2)}     \text{ for } k \geq 4,$$
$$\frac{m(E_6)}{3}= \tau(E_6)+\epsilon(E_6), \text{ where }\epsilon(E_6)=\frac{23}{24}, \  
\frac{m(E_7)}{3}= \tau(E_7)+\epsilon(E_7), \text{ where }\epsilon(E_7)=\frac{47}{48} $$ and 
$$\frac{m(E_8)}{3}= \tau(E_8)+\epsilon(E_8), \text{ where }\epsilon(E_8)=\frac{119}{120}.$$
\end{lemma}
Using Lemma \ref{L2}, Definition \ref{D2} can be formulated as follows.
\begin{corollary}
\label{C1}
Let $C \subset \mathbb{P}^{2}_{\mathbb{C}}$ be a reduced simply singular curve of even degree $n\geq 6$. Define $\tau(C):=\sum_{p \in {\rm Sing}(C)}\tau(C,p)$, the total Tjurina number of $C$, and $\epsilon(C):=\sum_{p \in {\rm Sing}(C)}\epsilon(C,p)$. Then $C$ is an MK-curve if and only if 
$$\tau(C)+\epsilon(C) = \frac{n}{6}(5n-6).$$
In particular, if $C$ is an MK-curve and $s(C)$ denotes the number of singularities of $C$, then
$$\tau(C)+\frac{s(C)}{2} \leq \frac{n}{6}(5n-6) < \tau(C) +s(C).$$
\end{corollary}
\begin{proof}
The last claim follows from the fact that
$$ \frac{1}{2} \leq \epsilon(C,p) < 1$$
for all singularities $p \in C$.
\end{proof}

Using Definition \ref{D2} and its reformulations above, one has the following results.
Our first result say that an MK-curve has necessarily a large number of singularities.
\begin{theorem}
\label{T1}
 If $C$ is an MK-curve of even degree $n=2m\geq 6$ and $s(C)$ denotes the number of singularities of $C$, then
 $$s(C) > \frac{m^2+3m-3}{3}.$$
 Moreover, if the curve $C$ has only $A_1$ and $A_2$ singularities, then the much stronger inequality
$$s(C) > \frac{7m^2-3m}{9}$$ 
holds. 
\end{theorem}
\proof
We use notation and results recalled in \cite[Section 2]{DP}. It was shown there, in the proof of Theorem 2.9, that the total Tjurina number $\tau(C)$ of a plane curve of degree $n=2m$ having only ${\rm ADE}$ singularities satisfies the inequality
$$\tau(C) \leq 3m(m-1)+1.$$
Using this result, Corollary \ref{C1} implies
$$s(C) > \frac{m}{3}(10m-6)- \tau(C) \geq \frac{m}{3}(10m-6)-3m(m-1)-1=
\frac{m^2+3m-3}{3}.$$
If the curve $C$ has only $A_1$ and $A_2$ singularities then, as in the proof of \cite[Theorem 2.9]{DP}, we get
$$\tau(C) \leq (2m-1)^2-r_0(2m-1-r_0)- \binom{2r_0+2-2m}{2},$$
where
$$r_0=\frac{5m}{3}-2.$$
It follows that
$$\tau(C) \leq \frac{23m^2-15m}{9}$$
and hence as above
$$s(C) > \frac{m}{3}(10m-6)- \tau(C) \geq \frac{m}{3}(10m-6)-\frac{23m^2-15m}{9}=
\frac{7m^2-3m}{9}.$$

\endproof
\begin{example}
\label{E2}
Theorem \ref{T1} applied to sextic curves gives that
$s(C) >5$. In other words, an MK-sextic must have at least 6 singularities. An MK-sextic having only nodes and cusps must have at least 7 singularities.
\end{example}
The following result shows that this bound is rather sharp, namely there are MK-sextics with 7 singularities.
\begin{proposition}
\label{P1}
The only line arrangement which is an MK-curve is the line arrangement which is given in suitable coordinates by the equation
$$(x^2-y^2)(y^2-z^2)(z^2-x^2)=0.$$
\end{proposition}
\proof
A line arrangement $\mathcal{L}$ has only simple singularities if it has only double and triple points. Let $n_2$ (resp. $n_3$) be the number of double points (resp. triple points) in $\mathcal{L}$. If the total number of lines in $\mathcal{L}$ is $2m$, then we have
$$n_2+3n_3=m(2m-1).$$
If we assume that $\mathcal{L}$ is an MK-curve, we get a new equation involving $n_2$ and $n_3$, namely
$$\frac{3n_2}{2}+\frac{39n_3}{8}=\frac{m(10m-6)}{3}.$$
Solving this system of two equations in $n_2$ and $n_3$, we get
$$n_2=\frac{9m-2m^2}{3}.$$
Since $n_2$ and $n_{3}$ are non-negative integers, it follows that $m=3$, $n_2=3$ and $n_3=4$, and this is the only non-negative integer solution. It is known that these numerical data corresponds to the line arrangement given in suitable coordinates by the equation
$$(x^2-y^2)(y^2-z^2)(z^2-x^2)=0.$$
\endproof

\begin{example}\label{E1}
Now we would like to discuss two additional examples of MK-curves having degree $6$. 
\begin{itemize}
\item It is well-known that a smooth elliptic curve $\mathcal{E}$ has exactly $9$ inflection points and the dual curve to a smooth elliptic curve is an irreducible sextic $\mathcal{S}$ having exactly $9$ simple cusps corresponding to the inflection points. Now we explain that $\mathcal{S}$ is an MK-curve. To this end, recall that for $p \in {\rm Sing}(\mathcal{S})$ we have $m(p)=8$, and thus
$$72 = \frac{n}{2}(5n-6) = \sum_{p \in {\rm Sing}(\mathcal{S})}m(p) = 9\cdot 8 =72.$$
\item There exists an arrangement $\mathcal{C}$ of one smooth conic and $4$ lines such that each line is tangent to the conic, and other singular points are just double intersection points -- we have $6$ such points. Using formulas for MK-numbers, we see that if $p \in \mathcal{C}$ is a double intersection point, then $m(p)=\frac{9}{2}$, and if $q \in {\rm Sing}(\mathcal{C})$ is a tacnode, then $m(q) = 11\frac{1}{4}$. Then we have
$$72 = \frac{n}{2}(5n-6) = \sum_{p\in {\rm Sing}(\mathcal{C})} m(p) = 11\frac{1}{4}\cdot 4 + \frac{9}{2}\cdot 6 = 72.$$
\end{itemize}

\end{example}
It turns out, as it was proved by Ivinskis' in his Master Thesis, the above mentioned $3$ types of MK-sextics are the only examples of MK-curves in degree $6$. 
\begin{definition}
\label{defI}
Let $C \subset \mathbb{P}^{2}_{\mathbb{C}}$ be a reduced sextic curve having only ${\rm ADE}$ singularities.
\begin{enumerate}
\item[(A)] We say that $C$ is of type $\mathfrak{A}$ if it consists of a smooth conic and four lines that these intersect the conic exactly at $4$ distinct points. Note that arrangements of type $\mathfrak{A}$ form one parameter family. 
\item[(B)] We say that $C$ is of type $\mathfrak{B}$ if this is an irreducible sextic curve with exactly $9$ simple cusps, so this is the dual curve to a smooth elliptic curve in $\mathbb{P}^{2}_{\mathbb{C}}$. Note that sextics of type $\mathfrak{B}$ form, up to the projective equivalence, one parameter family.
\item[(C)] We say that $C$ is of type $\mathfrak{C}$ if this is an arrangement of $6$ lines having exactly $4$ triple and $3$ double intersection points.
\end{enumerate}
\end{definition}
\begin{theorem}[Ivinskis]
\label{thmI}
A reduced curve $C \subset \mathbb{P}^{2}_{\mathbb{C}}$ is an MK-sextic if and only if $C$ is of type $\mathfrak{A}, \mathfrak{B}, \mathfrak{C}$.
\end{theorem}
For the proof, we refer to \cite[4.8.3 Theorem]{Ivinskis}.
\begin{remark}
As it was kindly pointed out by the referee, in the case of arrangements coming from families $\mathfrak{A}$ and $\mathfrak{B}$, the Kodaira dimension of the associated $4$-tuples $(X,\tilde{X},0,E)$ is $1$. Furthermore, the associated $K3$ surfaces to arrangements of type $\mathfrak{A}$ (respectively, of type $\mathfrak{B}$) are quotients of abelian surfaces by order $4$ (respectively, by order $3$) symplectic automorphisms.
\end{remark}

\section{MK-curves and dual curves}

Let $C \subset \mathbb{P}^{2}_{\mathbb{C}}$ be an irreducible curve of  degree $n$ having only isolated singularities. Then the dual curve 
$C^{\vee}$ is also irreducible, and its degree $n^{\vee}$ is given by the formula
\begin{equation}
\label{Eq1}
n^{\vee}=n(n-1)-\sum_{p \in {\rm Sing}(C)}(\mu(C,p)+\mult (C,p)-1),
\end{equation}
where $\mu(C,p)$ denotes the Milnor number of the singularity $(C,p)$ and $\mult (C,p)$ denotes its multiplicity, see for instance \cite[Formula (1.2.18)]{Dbook2}. Since $\mult (C,p)\geq 2$ for any singularity, it follows that for a plane curve having only simple singularities we have
\begin{equation}
\label{Eq2}
n^{\vee} \leq n(n-1)-\tau(C)-s(C),
\end{equation}
and equality holds if and only if  $C$ has only $A_k$ singularities.

\begin{corollary}
\label{C10}
Let $C \subset \mathbb{P}^{2}_{\mathbb{C}}$ be an irreducible MK-curve of degree $n=2m$. Then
$$n^{\vee} = \deg (C^{\vee}) <\frac{2m^2}{3}. $$
\end{corollary}
\proof
Using the inequality \eqref{E2} and Corollary \ref{C1}, we have
$$n^{\vee} \leq n(n-1)-\tau(C)-s(C)<2m(2m-1)-\frac{m(10m-6)}{3}=\frac{2m^2}{3}. $$
\endproof
In view of Ivinskis' Theorem \ref{thmI}, in the study of MK-curves we may assume
$n= \deg (C) \geq 8$. 
\begin{proposition}
\label{P10}
Let $C \subset \mathbb{P}^{2}_{\mathbb{C}}$ be an irreducible MK-curve of even degree $n\geq 8$. Assume that $C$ has $n_k$ singularities of type $A_k$ for $k=1,2$ and no other singularities. Then one has
$$n_2-n_1 \geq 20-n.$$
\end{proposition}
\proof
Note that formula \eqref{Eq1} implies that the dual of an irreducible curve of degree $\leq 3$ has degree $\leq 6$.
Hence in our case degree $n^{\vee} = \deg (C^{\vee})$ must be at least $4$ and we get the inequality
\begin{equation}
\label{Eq3}
2n_1+3n_2 \leq n(n-1)-4.
\end{equation}
On the other hand, the fact that $C$ is an MK-curve implies
\begin{equation}
\label{Eq4}
9n_1+16n_2 = n(5n-6).
\end{equation}
If we multiply the inequality \eqref{Eq3} by 5 and subtract the equality 
\eqref{Eq4}, we get
$$n_1-n_2\leq n-20,$$
which is exactly our claim. 
\endproof

\section{On the existence of MK-curves of degree greater than 6}
In this section we elaborate on the existence of MK curves having degree $n>6$. As we have seen in Section $3$, there is a complete classification of MK-sextics and we know a linkage between the geometry of MK-sextics and $K3$ surfaces by a result due to Kobayashi. The case of MK-curves having even degree $\geq 8$ is extremely interesting and important since such curves can lead to examples of ball-quotient surfaces $(X,\tilde{X},0,E)$. However, we are not aware of any example of MK-octic or MK-dodectic! We start with describing the world-record example of an octic curve having the highest known value of $m(C)$.
\begin{example}[Steiner quartic and 4 lines] Let us consider the Steiner quartic curve that is given by the following equation
$$F(x,y,z) = -\frac{1}{4}y^{2}x^{2}-z^{2}(x^{2}+y^{2}-2xy)+x^{2}yz + y^{2}xz.$$
This quartic is an irreducible plane curve having exactly $3$ cusps. We add now four lines, namely one bitangent line to the quartic, we can take the line at infinity $z=0$, and three tangents to cusps - notice that these three lines intersect at the ordinary triple point. Denote by $\mathcal{SQ}$ the resulting arrangement of $4$ lines and the Steiner quartic. We have altogether exactly three $E_7$ singularities, one $D_{4}$, two $A_{3}$, and six $A_{1}$ singular points. Observe that
$$m(\mathcal{SQ}) = 135,9375 < \frac{n}{2}(5n-6) = 136,$$
so $\mathcal{SQ}$ is not any MK-octic, but its value of $m(\mathcal{SQ})$ is the highest-known.
\end{example}
\begin{remark}
In the case of degree $10$ curves, we have the same problem, namely we are not aware of any MK-dodectic curve! 
\end{remark}
We start our discussion with the case of octics having only $A_{1}$, $A_{2}$, $A_{3}$, and $D_{4}$ singularities. Before that happen, we present a general result on reduced plane curves with simple singularities of type $A_{1}, A_{2}, A_{3}$, and $D_{4}$. 

In order to prove such a result, we are going to use Langer's version on the orbifold Bogomolov-Miyaoka-Yau inequality proved in \cite[Section 11.1]{Langer}.
\begin{theorem}[Langer]
\label{langer}
Let $C \subset \mathbb{P}^{2}_{\mathbb{C}}$ be a reduced curve of degree $n$ and assume that $(\mathbb{P}^{2}_{\mathbb{C}},\alpha C)$ is a an effective log canonical pair for a suitably chosen $\alpha\in [0,1]$, then one has
$$\sum_{p \in {\rm Sing}(C)} 3\bigg(\alpha(\mu_{p}-1)+1-e_{orb}(p,\mathbb{P}^{2}_{\mathbb{C}},\alpha C)\bigg)\leq (3\alpha - \alpha^{2})n^{2}-3\alpha n,$$
where ${\rm Sing}(C)$ denotes the set of all singular points, $\mu_{p}$ is the Milnor number of a singular point $p$, and $e_{orb}$ denotes the local orbifold Euler number of $p$.
\end{theorem}
\begin{theorem}
\label{BMYoct}
Let $C \subset \mathbb{P}^{2}_{\mathbb{C}}$ be a reduced curve of degree $n \geq 8$ having $n_{1}$ singularities of type $A_{1}$, $n_{2}$ singularities of type $A_{2}$, $n_{3}$ singularities of type $A_{3}$, and $d_{4}$ singularities of type $D_{4}$. Then
$$468n_{1} + 814n_{2}+1128n_{3}+1485d_{4} \leq 252n^{2}-288n.$$
\end{theorem}
\begin{proof}
 Consider the pair $(\mathbb{P}^{2}_{\mathbb{C}}, \alpha C)$, where $C$ has only singularities of type $A_{1},A_{2},A_{3}, D_{4}$. Recall that for our singularities we have the following data:
\begin{table}[h]
\centering
\begin{tabular}{c|c|c|c}
 Singularity Type   & $\mu_{p}$ & $e_{orb}(p,\mathbb{P}^{2}_{\mathbb{C}},\alpha C)$ & $\alpha$ \\ \hline \hline
$A_{1}$ & 1 & $(1-\alpha)^{2}$ & $0 < \alpha \leq 1$ \\
$A_{2}$ & 2 & $\frac{(5-6\alpha)^{2}}{24}$ & $1/6 \leq \alpha \leq 5/6$  \\
$A_{3}$ & 3 & $\frac{(3-4\alpha)^{2}}{8}$ & $1/4 \leq \alpha \leq 3/4$  \\
$D_{4}$ & 4 & $\frac{(2-3\alpha)^{2}}{4}$ & $0<\alpha \leq 2/3$.

\end{tabular}
\end{table}

Since our pair has to be effective and log canonical, one should have $\alpha \in [3/n, 2/3]$. Since $n\geq 8$, let us take $\alpha=\frac{3}{8}$. We are going to use Theorem \ref{langer} for the pair 
$(\mathbb{P}^{2}_{\mathbb{C}}, \frac{3}{8}C)$, namely
\begin{gather*}
\sum_{p \in {\rm Sing}(C)}3\bigg(\frac{3}{8} (\mu_{p}-1)+1 - e_{orb}\bigg(p, \mathbb{P}^{2}_{\mathbb{C}},\frac{3}{8} C \bigg)\bigg) \leq \frac{63}{64}n^{2} - \frac{9}{8} n. 
\end{gather*}
We start with the left-hand side. One has
\begin{gather*}3n_{1}\bigg(\frac{3}{8}(1-1)+1-\frac{25}{64}\bigg) + 3n_{2}\bigg(\frac{3}{8}(2-1)+1-\frac{121}{384}\bigg) + 3n_{3}\bigg(\frac{3}{8}(3-1)+1-\frac{9}{32}\bigg)  \\ + 3d_{4}\bigg(\frac{3}{8}(4-1)+1-\frac{49}{256}\bigg) = \frac{117}{64}n_{1} + \frac{407}{128}n_{2} + \frac{141}{32}n_{3} + \frac{1485}{256}d_{4}.
\end{gather*}
Then
$$ \frac{117}{64}n_{1} + \frac{407}{128}n_{2} + \frac{141}{32}n_{3} + \frac{1485}{256}d_{4} \leq  \frac{63}{64}n^{2}-\frac{9}{8}n.$$
Multiplying by $256$ we obtain the desired inequality, which completes the proof.
\end{proof}
Now we focus on the existence of MK-octics with $A_{1},A_{2},A_{3}$, and $D_{4}$ singularities. If $C \subset \mathbb{P}^{2}_{\mathbb{C}}$ is such an octic, then we have
\[\begin{cases}
36n_{1} + 64n_{2} + 90n_{3}+117d_{4} = 8\cdot m(10m-6) = 8\cdot 136 = 1088, \\
n_{1}+2n_{2}+3n_{3}+4d_{4} \leq 3m(m-1)+1 = 37,\\
n_{1},n_{2},n_{3},d_{4} \in \mathbb{Z}_{\geq 0}.
\end{cases}\]
It turns our that the above system has exactly $8$ solutions in non-negative integers, namely
\begin{enumerate}
\item[$a)$] $(n_{1},n_{2},n_{3},d_{4})=(0,17,0,0)$,
\item[$b)$] $(n_{1},n_{2},n_{3},d_{4})=(1,8,6,0)$,
\item[$c)$] $(n_{1},n_{2},n_{3},d_{4})=(2,8,3,2)$,
\item[$d)$] $(n_{1},n_{2},n_{3},d_{4})=(3,8,0,4)$,
\item[$e)$] $(n_{1},n_{2},n_{3},d_{4})=(6,8,4,0)$,
\item[$f)$] $(n_{1},n_{2},n_{3},d_{4})=(7,8,1,2)$,
\item[$g)$] $(n_{1},n_{2},n_{3},d_{4})=(11,8,2,0)$,
\item[$h)$] $(n_{1},n_{2},n_{3},d_{4})=(16,8,0,0)$.
\end{enumerate}
In the first step, we can automatically exclude case $a)$. It follows from Zariski's theorem that the maximal possible number of cusps for an octic is less than $16$, see \cite{Zar}. Then, using inequality proved in Theorem \ref{BMYoct}, we can exclude the cases $d)$, $e)$, $f)$, $g)$, and $h)$ - this is a straightforward check. It remains to check whether the cases $b)$ and $c)$ can be geometrically realized by a reduced plane curve of degree $8$. In order to decide whether the case $b)$ can hold, we mimic our argument as in Theorem \ref{BMYoct} for $\alpha=\frac{55}{100}$ obtaining the following result (that we treat as a corollary, mostly in order to avoid repetitions).
\begin{corollary}
In the setting of Theorem \ref{BMYoct}, we have
\begin{equation}
\label{bmyb}
3828n_{1}+6862n_{2}+9696n_{3}+12573d_{4} \leq 2156n^2-2640n.
\end{equation}
\end{corollary}
Now using (\ref{bmyb}) we can exclude the case $b)$ so the last remaining step is to decide on the case $c)$. Again, we mimic our argument as in Theorem \ref{BMYoct} for $\alpha=\frac{48}{100}$ obtaining the following result.
\begin{corollary}
In the setting of Theorem \ref{BMYoct}, we have
\begin{equation}
\label{bmyc}
10944n_{1}+19391n_{2}+27213n_{3}+35424d_{4} \leq 6048n^2-7200n.
\end{equation}
\end{corollary}
Using (\ref{bmyc}) we can exclude the case $c)$. The above discussion leads us to the following result.
\begin{theorem}
\label{thmn=8}
There does not exist any MK-octic with $A_{1}, A_{2},A_{3}$, and $D_{4}$ singularities. 
\end{theorem}
Based on the discussion presented above, we can observe the following corollary.
\begin{corollary}
\label{corn=10}
There is no MK-dodectic curve having nodes, tacnodes, and ordinary triple points. 
\end{corollary}
\begin{proof}
If follows from the fact that the Diophantine equation 
$$\frac{9}{2}n_{1} + 11\frac{1}{4}n_{3}+14\frac{5}{8}d_{4} = 220$$
does not have non-negative integer solutions.
\end{proof}

\begin{proposition}
\label{P11}
Let $C \subset \mathbb{P}^{2}_{\mathbb{C}}$ be an irreducible MK-curve of even degree $n=2m\geq 6$. Assume that $C$ has $n_k$ singularities of type $A_k$ for $k=1,2$ and no other singularities. Then one has
$$n_2> \frac{7m^2-21m}{18}+10.$$
In particular, there is no such curve if $n\in \{8,10,12\}$.
\end{proposition}

\proof
Using Theorem \ref{T1} and Proposition \ref{P10}, we have
$$2n_2 > \frac{7m^2-3m}{9}+20-2m$$
which is our first claim. 

When $m=4$, this gives $n_2 \geq 12$.
Note that
when $n=8$, equation \eqref{Eq4} implies that $n_1$ is divisible by 16.
If $n_1=0$, then equation \eqref{Eq4} implies that $n_2=17$, which is impossible due to Zariski's theorem \cite{Zar}. If $n_1=16$, then equation \eqref{Eq4} implies that $n_2=8$, impossible by our bound above. The cases $n_1 >16$ are discarded by the same argument.

When $m=5$, our bound gives $n_2 \geq 14$. Equation \eqref{Eq4} implies that $n_1=8(2k+1)$, with $k$ being a positive integer. When $n_1=8$,
equation \eqref{Eq4} implies that $n_2=27$, which is impossible, see \cite[Table on p. 25]{GS}. If $n_1=24$, then equation \eqref{Eq4} implies that $n_2=14$, a contradiction with Proposition \ref{P10}.
The cases $n_1 >24$ are discarded by the same argument or using the inequality $n_2 \geq 14$. \\
The same approach as above shows that when $m=6$ the only possibility would  be
$$ n_1=8 \text{ and } n_2=36.$$
Applying the Bogomolov-Miyaoka-Yau inequality for the pair $(\mathbb{P}^{2}_{\mathbb{C}},\frac{1}{4}C)$, where $C$ is a reduced curve of degree $n\geq12$ with only $A_{1}$ and $A_{2}$ singularities, we arrive at the inequality
\begin{equation}
42n_{1} + 71n_{2} \leq 22n^{2}-24n.
\end{equation}
Using this inequality for $n=12$, $n_{1}=8$, and $n_{2}=36$ we get a contradiction. 
\endproof

The above techniques can be used to show the non-existence of MK-curves having a single type of singular points. The following results are easy consequences of Theorem \ref{BMYoct}.

\begin{corollary}
\label{corA1}
Let $C \subset \mathbb{P}^{2}_{\mathbb{C}}$ be a  curve of even degree $n \geq 8$ which is nodal, i.e., $C$ has only $A_1$ singularities. Then $C$ is not an  MK-curve. 
\end{corollary}

\begin{corollary}
\label{corA2}
Let $C \subset \mathbb{P}^{2}_{\mathbb{C}}$ be a curve of even degree $n \geq 8$ which is cuspidal, i.e., $C$ has only $A_2$ singularities. Then $C$ is not an MK-curve. 
\end{corollary}

Now we focus on the case of curves with $E_{6}$ singularities. We start with the following.

\begin{theorem}
\label{thmE6}
There does not exist any MK curve $C$ of degree $n \in \{6,...,24\}$ with only $E_{6}$ singularities.
\end{theorem}
\begin{proof}
Assume that $C$ is an MK-curve with only $E_{6}$ singularities.  Denoting the number of $E_{6}$ singularities by $e_{6}$ and taking $n=2m$, the following condition holds 
\begin{equation}
\frac{167}{8}\cdot e_{6}=m(10m-6) = 10m^{2}-6m.
\end{equation}
Using Langer's orbifold Bogomolov-Miyaoka-Yau inequality for the pair $(\mathbb{P}^{2}_{\mathbb{C}},\frac{7}{12}C)$ we obtain
\begin{gather*}
3e_{6}\cdot \bigg(\frac{35}{12}+1\bigg) = \sum_{p \in {\rm Sing}(C)}3\bigg(\alpha(\mu_{p}-1)+1 - e_{orb}(p, \mathbb{P}^{2}_{\mathbb{C}},\alpha C)\bigg) \\ \leq (3\alpha - \alpha^{2})n^{2} - 3\alpha n = \frac{203}{144}n^2 - \frac{21}{12}n,
\end{gather*}
where $\mu_{p}$ is the local Milnor number of $p$, and $e_{orb}$ denotes the local orbifold Euler number of $p \in {\rm Sing}(C)$. Observe that for $\alpha =\frac{7}{12}$ the local orbifold Euler number is equal to $0$ and this follows from the fact that $\alpha=\frac{7}{12}$ is the log canonical threshold for $E_{6}$ singularities. This gives us the following upper-bound:
\begin{equation}
e_{6}\leq \frac{203}{1692}n^{2}-\frac{7}{47}n.
\end{equation}
By the above considerations, we arrive at the following inequality for every $m\in \{3, ..., 12\}$
\begin{equation}
\frac{203}{1692}\cdot(2m)^{2}-\frac{7}{47}\cdot(2m) \geq e_{6} = \frac{80}{167}m^{2}-\frac{48}{167}m,
\end{equation}
which means that 
$$m\geq \frac{738}{61}\approx 12.10,$$
a contradiction.
\end{proof}
The above result is rather technical, but also sharp as we can see right now.
\begin{example}
\label{exB}
In \cite[Example 3.4]{Bonnafe} Bonnaf\'e proved that there exists a reduced curve $C_{18}$ of degree $n=18$ with exactly $36$ singularities of type $E_{6}$. It turns our that $C_{18}$ is optimal in the sense that it reaches the upper-bound on the number of $E_{6}$ singularities, which can be verified using the above theorem. Observe that
$$\frac{n}{2}(5n-6)- \frac{167}{8}\cdot e_{6} = \frac{9}{2},$$
which means that $C_{18}$ misses the property of being an MK-curve by a node!
\end{example}
\begin{remark}
\label{rkB}
Among others, Bonnaf\'e explained in \cite{Bonnafe} that there exists a reduced curve $C_{18}'$ of degree $n=18$ with $72$ singularities of type $A_{2}$ and $12$ singularties of type $D_{4}$. Denoting by $n_2$ the number of $A_{2}$ singularities and by $d_{4}$ the number of $D_{4}$ singularities, observe that
$$\frac{n}{2}(5n-6) - 8\cdot n_{2} - 14\frac{5}{8}\cdot d_{4} = \frac{9}{2},$$
and unluckily $C_{18}'$ misses the property of being an MK-curve by a node.
\end{remark}
\begin{remark}
Using the same arguments one can show that there does not exist any MK-curve of degree $n \in \{6, ..., 118\}$ with only $E_{7}$ singularities. Unfortunately there is not direct argument which allows us to exclude the existence of such MK-curves with only $E_{7}$ singularities for every even degree $n\geq 6$.
\end{remark}
Now we look at MK-curves with only $E_{8}$ singularities.
\begin{theorem}
\label{thmE8}
There does not exist any MK curve $C$ of degree $n\geq 6$ with only $E_{8}$ singularities.
\end{theorem}
\begin{proof}
We are going to mimic the argument presented for MK-curves with only $E_{6}$ singularities. Assume that $C$ is an MK-curve of degree $n=2m\geq 6$ with only $E_{8}$ singularities. We have then
\begin{equation}
\frac{1079}{40}\cdot e_{8}=m(10m-6) = 10m^{2}-6m,
\end{equation}
where $e_{8}$ denotes the number of $E_{8}$ singularities. Using Langer's orbifold Bogomolov-Miyaoka-Yau inequality for the pair $(\mathbb{P}^{2}_{\mathbb{C}},\frac{8}{15}C)$, where $\alpha=\frac{8}{15}$ is the log canonical threshold for $E_{8}$ singularities, we obtain the following upper-bound:

\begin{equation}
e_{8}\leq \frac{296}{3195}n^{2}-\frac{8}{71}n.
\end{equation}
Combining this inequality with the condition on MK-curves with only $E_{8}$ singularities, we arrive at
\begin{equation}
\frac{1184}{3195} m^{2}-\frac{16}{71} m \geq e_{8} = \frac{400}{1079}m^{2}-\frac{240}{1079}m,
\end{equation}
so we have a contradiction for every $m\geq 3$.
\end{proof}

\section{On the freeness of simply singular plane curves}

 Let $C$ be a reduced curve $\mathbb{P}^{2}_{\mathbb{C}}$ of degree $n$ given by $f \in S :=\mathbb{C}[x,y,z]$. We denote by $J_{f}$ the Jacobian ideal generated by the partials derivatives $\partial_{x}f, \, \partial_{y}f, \, \partial_{z}f$. We define  the \textbf{minimal degree of relations} among the partial derivatives, denoted here by ${\rm mdr}(f)$, which is equal to the minimal degree $r$ of a non-trivial triple $(a,b,c) \in S_{r}^{3}$ such that 
$$a\cdot \partial_{x} f + b\cdot \partial_{y}f + c\cdot \partial_{z}f = 0.$$
We denote by $\mathfrak{m} = \langle x,y,z \rangle$ the irrelevant ideal. Consider the graded $S$-module $N(f) = I_{f} / J_{f}$, where $I_{f}$ is the saturation of $J_{f}$ with respect to $\mathfrak{m}$. It is known that $N(f)$ is an Artinian graded $S$-module, in particular the integer
$$\nu(C)=\max _k \dim N(f)_k$$
is well defined.
\begin{definition}
The integer $\nu(C)$ is called the {\it freeness defect} of the reduced complex plane curve $C$.
\end{definition}
This definition is justified by the fact that the curve $C$ is free (resp. nearly free) if and only if $\nu(C)=0$ (resp. $\nu(C)=1$). For more on free and nearly free curves we refer to \cite{DP,DimStic}. In particular, we have shown in  \cite[Theorem 2.9]{DP} that the maximizing curves in the study of surfaces with maximal Picard number form a special class of free curves.
In this section we study the freeness defect of the MK-curves. We start with a general result applying to all simply singular curves of even degree.
\begin{theorem}
\label{thmNU}
Let $C$ be a reduced curve $\mathbb{P}^{2}_{\mathbb{C}}$ of even degree $n=2m$ having only ${\rm ADE}$ singularities. Then the freeness defect of $C$ is given by the following equality:
$$\nu(C)= 3m^2-3m+1 -\tau (C).$$
\end{theorem}
\proof
We have shown in the proof of \cite[Theorem 2.9]{DP} that for a plane curve $C:f=0$ of degree $n=2m$ having only ${\rm ADE}$ singularities the minimal degree of a relation among the partial derivatives satisfies the inequality
$${\rm mdr}(f) \geq m-1.$$
In other words, we have
$${\rm mdr}(f) \geq \frac{n-2}{2}$$
and our claim follows from \cite[Theorem 1.2 (2)]{D_CMH} by a direct computation.
\endproof
\begin{corollary}
\label{corNU}
Let $C$ be a reduced MK-curve $\mathbb{P}^{2}_{\mathbb{C}}$ of even degree $n=2m\geq 6$. Then the freeness defect of $C$ is given by the following equality.
$$\nu(C)= \epsilon(C)+1-\frac{m(m+3)}{3}.$$
In particular, one has in this situation
$$\epsilon(C)\geq \frac{m(m+3)}{3}-1,$$
where $\epsilon(C)$ is the constant defined in Corollary \ref{C1}.
\end{corollary}
\proof It follows from Theorem \ref{thmNU} and Corollary \ref{C1}.
\endproof

\begin{example}
\label{exNU}
Consider now the three types of MK-sextics $C$ described in Definition \ref{defI} and Example \ref{E1}. When $C$ is of type $\mathfrak{A}$, then $C$ has $6$ singularities of type $A_{1}$ and $4$ singularities of type $A_3$. It follows that in this case
$$\nu(C)= 19-\tau(C)=19-6-4\cdot 3=1,$$
and hence $C$ is nearly free. When $C$ is of type $\mathfrak{B}$, then $C$ has $9$ singularities of type $A_2$. We get
$$\nu(C)= 19-\tau(C)=19-9 \cdot 2=1,$$
and hence again this sextic $C$ is nearly free. When $C$ is of type $\mathfrak{C}$, then $C$ has $3$ singularities of type $A_{1}$ and $4$ singularities of type $D_{4}$. It follows that in this case
$$\nu(C)= 19-\tau(C)=19-3-4\cdot 4=0,$$
and hence in this case $C$ is free.
\end{example}

\begin{example}
\label{exNU2}
Consider now the curve $C_{18}$ discussed in Example \ref{exB} above.
This curve has degree $n=18$ and has $36$ singularities of type $E_{6}$. It follows that
$$\nu(C_{18})= 217-\tau(C)=217-36 \cdot 6=1,$$
and hence $C_{18}$ is nearly free. We look now at the curve $C'_{18}$ discussed in Remark \ref{rkB} above.
This curve has degree $n=18$ and has $72$ singularities of type $A_2$ and $12$ singularities of type $D_4$. It follows that
$$\nu(C'_{18})= 217-\tau(C)=217-72 \cdot 2 - 12 \cdot 4=25,$$
and hence $C'_{18}$ is very far of being a free curve. Since both $C_{18}$ and $C'_{18}$ are close to being MK-curves, this example may suggest that the property of being an MK-curve is not related to some freeness property, in spite of the case of MK-sextics discussed in Example \ref{exNU} above.
\end{example}

\section*{Acknowledgments}
We would like to thank an anonymous referee for valuable comments that allowed us to improve the paper. 

Alexandru Dimca was partially supported by the Romanian Ministry of Research and Innovation, CNCS - UEFISCDI, Grant \textbf{PN-III-P4-ID-PCE-2020-0029}, within PNCDI III.

Piotr Pokora as partially supported by the National Science Center (Poland) Sonata Grant Nr \textbf{2018/31/D/ST1/00177}.

\vskip 0.5 cm


\begin{thebibliography}{000}
\bibitem{Bonnafe}
C. Bonnaf\'e, Some singular curves and surfaces arising from invariants of complex reflection groups. \textit{Exp. Math.} \textbf{30(3)}: 429 -- 440 (2021); correction ibid. 29, No. 3, 360 (2020).

\bibitem{Dbook2}
A. Dimca, {\em Singularities and topology of hypersurfaces}, Universitext, Springer-Verlag, New York, 1992. 

\bibitem{D_CMH}
A. Dimca,  On rational cuspidal plane curves and the local cohomology of Jacobian rings. \textit{Comment. Math. Helv.} \textbf{94}: 689 -- 700 (2019). 

\bibitem{DP}
A. Dimca and P. Pokora, Maximizing curves viewed as free curves. \textit{International Mathematical Research Notes} \textbf{(22)}: 19156 -- 19183 (2023),


\bibitem{DimStic} A. Dimca and G. Sticlaru, Free and Nearly Free Curves vs. Rational Cuspidal Plane Curves. \textit{Publ. Res. Inst. Math. Sci.} \textbf{54(1)}: 163 -- 179 (2018). 

\bibitem{duP}
A. A. Du Plessis and C. T. C. Wall, Application of the theory of the discriminant to highly singular plane curves. \textit{Math. Proc. Camb. Philos. Soc.} \textbf{126(2)}: 259 -- 266 (1999).

\bibitem{GS} G.-M. Greuel and E. Shustin, Plane Algebraic Curves with Prescribed Singularities. In: Cisneros-Molina, J.L., Lê, D.T., Seade, J. (eds) Handbook of Geometry and Topology of Singularities II. Springer, Cham. \url{https://doi.org/10.1007/978-3-030-78024-1_2} (2021).

\bibitem{Hirzebruch}
F. Hirzebruch, Singularities of algebraic surfaces and characteristic numbers. Algebraic geometry, Proc. Lefschetz Centen. Conf., Mexico City/Mex. 1984, Part I, \textit{Contemp. Math.} \textbf{58}: 141 -- 155 (1986).

\bibitem{Ivinskis}
K. Ivinskis, Normale Fl\"{a}chen und die Miyaoka-Kobayashi Ungleichung. Diplomarbeit, Bonn, (1985).

\bibitem{Koba}
R. Kobayashi, Lectures at the Max Planck Institute in Bonn, May 1985 (unpublished notes).
\bibitem{Kob1}
R. Kobayashi, S. Nakamura, F. Sakai, A numerical characterization of ball quotients for normal surfaces with branch loci. \textit{Proc. Japan Acad., Ser. A} \textbf{65(7)}: 238 -- 241 (1989).
\bibitem{Kob2}
R. Kobayashi, Uniformization of complex surfaces. \textit{Kähler metric and moduli spaces, Adv. Stud. Pure Math.} \textbf{18(2)}: 313 -- 394 (1990).
\bibitem{Langer}
A. Langer, Logarithmic orbifold Euler numbers of surfaces with applications. \textit{Proc. London Math. Soc.} \textbf{86}: 358 -- 396 (2003).
\bibitem{Zar}
O. Zariski, On the non-existence of curves of order 8 with 16 cusps. \textit{Am. J. Math.} \textbf{53}: 309 -- 318 (1931).
\end{thebibliography}
\end{document}